\documentclass[12pt]{amsart}
\usepackage{amssymb,amsmath}
\usepackage{graphicx}
\def\R{{\mathbb R}}

\def\I{{\mathfrak I}}

\def\e{{\rm e}}
\def\d{{\rm d}}
\def\i{{\rm i}}
\def\m{{\mathfrak n}}

\def\AA{{\mathbb{A}}}

\newtheorem{proposition}{Proposition}
\newtheorem{theorem}[proposition]{Theorem}

\newtheorem{lemma}[proposition]{Lemma}
\theoremstyle{definition}

\newtheorem*{remark}{Remark}
\newtheorem*{notation}{Notation}
\def \au {\rm}
\def \ti {\it}
\def \jou {\rm}
\def \bk {\it}
\def \no#1#2#3 {{\bf #1} (#3), #2.}
\def \eds#1#2#3 {#1, #2, #3.}
\title[A quantitative Riemann-Lebesgue lemma]
{A quantitative Riemann-Lebesgue lemma\\
with application to equations with memory}
\author[F. Dell'Oro, E. Laeng and V. Pata]{Filippo Dell'Oro, Enrico Laeng and Vittorino Pata}

\address{Politecnico di Milano - Dipartimento di Matematica
\newline\indent
Via Bonardi 9, 20133 Milano, Italy}
\email{filippo.delloro@polimi.it {\rm (F. Dell'Oro)}}
\email{enrico.laeng@polimi.it {\rm (E. Laeng)}}
\email{vittorino.pata@polimi.it {\rm (V. Pata)}}

\subjclass[2010]{42A38, 37L15, 45K05}
\keywords{Half Fourier transform, Riemann-Lebesgue lemma, equations with memory, semiuniform stability,
decay rates}
\begin{document}

\begin{abstract}
An elementary proof of a quantitative version of the Riemann-Lebesgue lemma for functions supported
on the half line is given. Applications to differential models with memory are discussed.
\end{abstract}

\maketitle

%%%%%%%%%%%%%%%%%%%%%%%%%%%%%%%%%%%%%%%%%%%%%%%%%%%%%%%%%%%%%%

\section{Statement of the Result}

\noindent
Let $f\in L^1(\R^+)$ be a real function, absolutely continuous on $\R^+=(0,\infty)$, and
whose derivative $f'$ is summable in a neighborhood of infinity (hence
$f'\in L^1(x,\infty)$ for every $x>0$).
The aim of this note is to provide an elementary proof of the asymptotic expansion as $\lambda\to\infty$
of the half Fourier transform
$$\hat f(\lambda)=\int_0^\infty \e^{-\i \lambda s} f(s) \,\d s$$
depending only on the behavior of $f$ in a neighborhood of zero, establishing a quantitative version
of the Riemann-Lebesgue lemma for functions supported on the half line.
The analogue for the Laplace transform is the well-known initial value theorem (see e.g.\ \cite{WID}).
For the Fourier counterpart, however, we could not locate a precise reference, although
similar results are mentioned in \cite{BGT}.
Such an asymptotic expansion turns out to be extremely useful,
for instance, when dealing with differential equations
containing convolution terms, as shown in the last part of the work.

\begin{notation}
For every $p\in[0,1)$, we define the number
$$\m(p)=-\i\e^{\frac{\i p\pi}2}\Gamma(1-p),$$
where
$\Gamma(x)=\int_0^\infty s^{x-1}\e^{-s}\,\d s$
is the Euler Gamma-function.
\end{notation}

\begin{theorem}
\label{UNO}
Let $p\in[0,1)$. Assume that
\begin{equation}
\label{EE}
\lim_{s \to 0} s^p f(s)=\ell\in\R
\end{equation}
and
\begin{equation}
\label{AA}
\limsup_{x \to 0} x^p \int_x^\infty |f'(s)|\,\d s<\infty.
\end{equation}
Then
$$\lim_{\lambda \to \infty} \lambda^{1-p}
\hat f(\lambda)= \ell\,\m(p).
$$
\end{theorem}

\noindent
$\diamond$
If $p=0$, assumption \eqref{AA} merely means $f'\in L^1(\R^+)$.
Since $f$ is summable, this is always the case
when $f$ is monotone and \eqref{EE} holds.

\smallskip
\noindent
$\diamond$
If $p>0$, it is easily seen that \eqref{AA} is satisfied whenever the function
$s\mapsto s^{1+p}f'(s)$ is bounded in a neighborhood of zero, or whenever
$f$ is monotone in a neighborhood of zero and \eqref{EE} holds.

\section{Proof of Theorem \ref{UNO}}

\noindent
Two general lemmas will be needed.

\begin{lemma}
\label{L2}
Let $p\in(0,1)$. Then for any $\lambda,\beta>0$, we have the equality
$$\lambda^{1-p}
\int_0^{\frac\beta\lambda} \frac{\e^{-\i \lambda s}}{s^p}\, \d s = \m(p)
+\frac{\i\e^{-\i \beta}}{\beta^p}-\frac{\i p}{\beta^p} \int_1^{\infty } \frac{\e^{-\i \beta s}}{s^{1+p}}\,\d s
$$
\end{lemma}

\begin{proof}
Via complex integration of the function $g(z)=z^{-p}\e^{-z}$ along the first quarter of the circumference of radius $R\to\infty$
(removing a small portion about zero),
one draws the equality
$$\int_0^{\infty } \frac{\e^{-\i s}}{s^p} \,\d s=\m(p).
$$
Hence, splitting the integral and
performing a change of variables, we are led to
$$
\lambda^{1-p}
\int_0^{\frac\beta\lambda} \frac{\e^{-\i \lambda s}}{s^p}\, \d s
=\m(p)
-\beta^{1-p} \int_1^{\infty } \frac{\e^{-\i \beta s}}{s^p}\, \d s.
$$
An integration by parts completes the argument.
\end{proof}

\begin{lemma}
\label{L1}
For any $\lambda,\alpha> 0$, we have the equality
\begin{align*}
\lambda^{1-p}
\int_{\alpha}^\infty\e^{-\i \lambda s} f(s)\,\d s
&=\frac{\lambda^{1-p}}2\int_\alpha^{\alpha+\frac\pi\lambda} \e^{-\i \lambda s}f(s)\, \d s
+\frac{\i \e^{-\i \lambda \alpha}}{2\lambda^p} \big[f(\alpha+{\textstyle\frac\pi\lambda})-f(\alpha)\big]\\
&\quad-\frac{\i}{\lambda^p} \int_\alpha^\infty \e^{-\i \lambda s}f'(s)\,\d s
+\frac{\i}{2\lambda^p} \int_\alpha^{\alpha+\frac\pi\lambda} \e^{-\i \lambda s}f'(s)\,\d s.
\end{align*}
\end{lemma}

\begin{proof}
By direct calculations, we derive the identity
$$\int_{\alpha}^\infty\e^{-\i \lambda s} f(s)\,\d s
=\frac{1}2\int_\alpha^{\alpha+\frac\pi\lambda} \e^{-\i \lambda s}f(s)\, \d s
-\frac12\int_{\alpha}^\infty\e^{-\i \lambda s} \bigg(\int_{s}^{s+\frac\pi\lambda}f'(\sigma)\,\d \sigma\bigg)\,\d s,
$$
and the conclusion follows from the Fubini theorem.
\end{proof}

Besides, assuming~\eqref{EE} true for some $p\in[0,1)$, we introduce the nondecreasing function
vanishing at zero
$$\omega_p(s)=\sup_{t\in(0,s)}\big|t^p f(t)-\ell\big|.$$
We shall treat separately two cases.

\subsection*{Case $\boldsymbol{p>0}$}
Define the function $\beta$ on $\R^+$ as follows:
setting $\lambda_0=0$,
select a strictly increasing sequence $\lambda_n\geq n^2$ in such a way that
$$\omega_p\big({\textstyle \frac{n}{\lambda_n}}\big)\leq \frac1{n^2}.
$$
Then, put
$$\beta(\lambda)=n\quad\text{for }\,\lambda\in [\lambda_n,\lambda_{n+1}).
$$
By construction, the (nonnegative) function $\beta$ is nondecreasing with
$$\lim_{\lambda\to\infty} \beta(\lambda)=\infty.$$
Besides
$$\lim_{\lambda\to\infty}\beta(\lambda)\omega_p\big({\textstyle \frac{\beta(\lambda)}{\lambda}}\big)=0,
$$
and
$$
\lim_{\lambda\to\infty}\frac{\beta(\lambda)}{\lambda}=0.$$
Indeed, if $n\geq 1$ and $\lambda\in [\lambda_n,\lambda_{n+1})$, we have that
$$\beta(\lambda)\omega_p\big({\textstyle \frac{\beta(\lambda)}{\lambda}}\big)=n\omega_p\big({\textstyle \frac{n}{\lambda}}\big)
\leq n\omega_p\big({\textstyle \frac{n}{\lambda_n}}\big)\leq \frac1n,
$$
and
$$\frac{\beta(\lambda)}{\lambda}=\frac{n}{\lambda}\leq \frac{n}{\lambda_n}\leq\frac1n.
$$
At this point, we write
$$\lambda^{1-p}
\hat f(\lambda)=
\I_1(\lambda)+\I_2(\lambda)+\I_3(\lambda),$$
where
\begin{align*}
\I_1(\lambda) &=\ell\lambda^{1-p}
\int_0^{\frac{\beta(\lambda)}\lambda}\frac{\e^{-\i \lambda s}}{s^p}\,\d s,\\
\I_2(\lambda) &=\lambda^{1-p}
\int_0^{\frac{\beta(\lambda)}\lambda}\e^{-\i \lambda s}\, \frac{s^p f(s)-\ell}{s^p} \,\d s,\\
\noalign{\vskip2mm}
\I_3(\lambda) &=\lambda^{1-p}
\int_{\frac{\beta(\lambda)}\lambda}^\infty\e^{-\i \lambda s} f(s)\,\d s.
\end{align*}

\noindent
$\bullet$ From Lemma \ref{L2} it is clear that
$$\lim_{\lambda\to\infty} \I_1(\lambda)=\ell\,\m(p).
$$

\smallskip
\noindent
$\bullet$ By direct calculations, for all $\lambda\geq\lambda_1$,
$$|\I_2(\lambda)|\leq \lambda^{1-p} \omega_p\big({\textstyle \frac{\beta(\lambda)}{\lambda}}\big)
\int_0^{\frac{\beta(\lambda)}\lambda}\frac1{s^p}\,\d s
\leq \frac1{1-p} \beta(\lambda)\omega_p\big({\textstyle \frac{\beta(\lambda)}{\lambda}}\big)\to 0
$$
as $\lambda\to\infty$.

\smallskip
\noindent
$\bullet$ From Lemma \ref{L1} with $\alpha=\frac{\beta(\lambda)}\lambda$, we get
\begin{align*}
\I_3(\lambda)
&=\frac{\lambda^{1-p}}2\int_{\frac{\beta(\lambda)}\lambda}^{\frac{\beta(\lambda)}\lambda+\frac\pi\lambda} \e^{-\i \lambda s}f(s)\, \d s
+\frac{\i \e^{-\i \beta(\lambda)}}{2\lambda^p}
\Big[f\big({\textstyle \frac{\beta(\lambda)}\lambda+\frac\pi\lambda}\big)-f\big({\textstyle\frac{\beta(\lambda)}\lambda}\big)\Big]\\
&\quad-\frac{\i}{\lambda^p} \int_{\frac{\beta(\lambda)}\lambda}^\infty \e^{-\i \lambda s}f'(s)\,\d s
+\frac{\i}{2\lambda^p} \int_{\frac{\beta(\lambda)}\lambda}^{\frac{\beta(\lambda)}\lambda
+\frac\pi\lambda} \e^{-\i \lambda s}f'(s)\,\d s.
\end{align*}
We now estimate the four terms in the right-hand side.
For $\lambda$ sufficiently large,
\begin{align*}
\frac{\lambda^{1-p}}2\bigg|\int_{\frac{\beta(\lambda)}\lambda}^{\frac{\beta(\lambda)}\lambda+\frac\pi\lambda}
\e^{-\i \lambda s}f(s)\, \d s\bigg|
&\leq |\ell|\lambda^{1-p}\int_{\frac{\beta(\lambda)}\lambda}^{\frac{\beta(\lambda)}\lambda+\frac\pi\lambda} \frac1{s^p}\, \d s\\
&=\frac{|\ell|}{1-p}[\beta(\lambda)]^{1-p}\Big[\big(1+{\textstyle\frac\pi{\beta(\lambda)}}\big)^{1-p}-1\Big],
\end{align*}
which, recalling that $\beta(\lambda)\to\infty$, converges to zero as $\lambda\to\infty$.
Concerning the second term, its modulus is readily seen to be bounded by
$$\frac{1}{2\lambda^p}
\Big[\big|f\big({\textstyle \frac{\beta(\lambda)}\lambda+\frac\pi\lambda}\big)\big|
+\big|f\big({\textstyle\frac{\beta(\lambda)}\lambda}\big)\big|\Big]\sim \frac{|\ell|}{[\beta(\lambda)]^{p}}\to 0.$$
Finally, the two remaining terms are controlled by
$$\frac{3}{2\lambda^p} \int_{\frac{\beta(\lambda)}\lambda}^\infty|f'(s)|\,\d s
=\frac{3}{2[\beta(\lambda)]^p} \bigg(\frac{\beta(\lambda)}\lambda\bigg)^p
\int_{\frac{\beta(\lambda)}\lambda}^\infty|f'(s)|\,\d s\to 0
$$
on account of \eqref{AA}.
This finishes the proof of the case $p>0$.
\qed

\subsection*{Case $\boldsymbol{p=0}$}
Extending by continuity $f$ at zero, using Lemma \ref{L1} and letting then $\alpha\to 0$ we get
\begin{align*}
\lambda\hat f(\lambda)
&=\frac{\lambda}2\int_0^{\frac\pi\lambda} \e^{-\i \lambda s}f(s)\, \d s
+\frac{\i}{2} \big[f({\textstyle\frac\pi\lambda})-\ell\big]\\
&\quad-\i\int_0^\infty \e^{-\i \lambda s}f'(s)\,\d s
+\frac{\i}{2} \int_0^{\frac\pi\lambda} \e^{-\i \lambda s}f'(s)\,\d s.
\end{align*}
It is clear from the continuity of $f$ and the Riemann-Lebesgue lemma
that the last three terms in the right-hand side above go to zero as $\lambda\to\infty$.
Moreover, as $\m(0)=-\i$,
$$\frac{\lambda}2\int_0^{\frac\pi\lambda} \e^{-\i \lambda s}f(s)\, \d s
=\ell\,\m(0) +\frac{\lambda}2\int_0^{\frac\pi\lambda} \e^{-\i \lambda s}\big[f(s)-\ell]\, \d s,$$
and
$$\bigg|\frac{\lambda}2\int_0^{\frac\pi\lambda} \e^{-\i \lambda s}\big[f(s)-\ell]\, \d s\bigg|
\leq\frac{\pi}2 \omega_0\big({\textstyle \frac{\pi}{\lambda}}\big)\to 0
$$
when $\lambda\to\infty$.
\qed
%%%%%%%%%%%%%%%%%%%%%%%%%%%%%%%%%%%%%%%%%%%%%

%%%%%%%%%%%%%%%%%%%%%%%%%%%%%%%%%%%%%%%%%%%%%
\section{Application to Equations with Memory}

\noindent
In this second part,
we discuss an application to equations with memory.
To this end, let
$A$ be a strictly positive unbounded linear operator on
a separable real Hilbert space,
with inverse $A^{-1}$ not necessarily compact.
We consider the linear equation with memory
$$
\partial_{tt} u (t) + A u(t)+\int_0^\infty
\mu(s)[u(t)-u(t-s)]\d s =0,
$$
where
$u(0)$ and $\partial_t u(0)$, as well as the past history $u(-s)_{|s> 0}$,
are assigned initial data.
The nonnegative function $\mu\not\equiv 0$, called memory kernel, is supposed to be nonincreasing, absolutely
continuous and summable on $\R^+$ (hence $\mu'$ is summable at infinity).
A concrete version of the equation, where $A$ is a particular realization of the Laplacian,
arises
in the theory of hereditary
electromagnetism, and serves as a model for the
evolution of the
electromagnetic field in the ionosphere (see~\cite{GNPbis} for more details).

The well-posedness of the related Cauchy problem and the asymptotic behavior of the solutions have been studied
in~\cite{GNPbis,RNV}.
There, by introducing in the same spirit
of~\cite{DAF} the auxiliary ``memory variable", the equation is shown to generate a contraction semigroup
$S(t)= \e^{t\mathbb{\AA}}$ acting
on a suitable Hilbert space accounting for the presence of the memory. Besides,
such a semigroup fails to be exponentially stable.
Nevertheless, if the inclusion
$\i\R\subset \rho(\AA)$ holds\footnote{$\rho(\AA)$ denotes the resolvent set of the (complexification of the)
infinitesimal generator $\AA$.}, in the light of the works
of Batty and coauthors~\cite{AB,Bat}
the semigroup $S(t)$ turns out to be {\it semiuniformly stable}, i.e.\
$$
\lim_{t\to\infty}\| S(t) \AA^{-1}\|= 0.
$$
In addition, the very same calculations needed to prove the
lack of exponential stability provide
an estimate from below
of the norm of the resolvent operator $(\i\lambda - \AA)^{-1}$.
Indeed, the following result holds (see \cite{DD,RN}).

\begin{theorem}
\label{DAN}
Assume that $\mu$ satisfies for some $\delta>0$ the classical condition~\cite{DAF}
$$
\mu'(s)+\delta\mu(s)\leq 0,\quad\forall s\in\R^+.
$$
Then
$\i\R\subset \rho(\AA)$, implying that $S(t)$ is semiuniformly stable. Besides,
$$
\limsup_{\lambda\to\infty}\frac{|\hat{\mu}(\lambda)|}{\lambda}\|(\i\lambda - \AA)^{-1}\|>0,
$$
where $\hat \mu$ is the half Fourier transform of $\mu$.
\end{theorem}

The latter estimate turns out to be a crucial tool in establishing the optimal decay rate,
for it is known from~\cite{Bat} that if
$$\limsup_{\lambda\to\infty} \lambda^{-\alpha}\|(\i\lambda - \AA)^{-1}\|>0$$
for some $\alpha>0$, then $\|S(t)\AA^{-1}\|$ cannot decay faster than $t^{-\frac1{\alpha}}$.
Thus, the problem of finding sharp polynomial growth rates at infinity
of $\hat \mu$ becomes paramount.
As an immediate consequence of Theorem \ref{UNO}, we have a
notable ``quantitative version" of Theorem \ref{DAN}.

\begin{theorem}
\label{NS}
Within the hypotheses of Theorem~\ref{DAN}, assume there exist $p\in[0,1)$ and $\ell>0$ such that
$$
\lim_{s\to 0} s^p \mu(s)=\ell.
$$
Then
$$
\limsup_{\lambda\to\infty}\lambda^{p-2}\|(\i\lambda - \AA)^{-1}\|>0.
$$
Accordingly, $\|S(t)\AA^{-1}\|$ cannot decay faster than $t^{-\frac1{2-p}}$.
\end{theorem}

\begin{remark}
In \cite{DD,RN} it is also shown that if the first limit above occurs for $p=0$, then $\|S(t)\AA^{-1}\|$
decays {\it at least} as $t^{-\frac1{2}}$. In which case, we conclude from Theorem~\ref{NS} that $t^{-\frac1{2}}$
is actually the optimal decay rate.
\end{remark}

The fact that Theorem~\ref{UNO} has not been considered in the earlier PDE literature is exemplified,
for instance, by the works \cite{DD,FSRQ,GRS,RN},
where the above quantitative estimates are carried out only for particular kernels of the form
$$
\mu(s) =s^{-p}\e^{-\delta s},
$$
with $\delta>0$ and $p\in[0,1)$, for which $\hat \mu$
can be explicitly computed.
%%%%%%%%%%%%%%%%%%%%%%%%%%%%%%%%%%%%%%%%%%%%%%%%%

%%%%%%%%%%%%%%%%%%%%%%%%%%%%%%%%%%%%%%%%%%%%%%%%%

%%%%%%%%%%%%%%%%%%%%%%%%%%%%%%%%%%%%%%%%%%%%%%%%%

%%%%%%%%%%%%%%%%%%%%%%%%%%%%%%%%%%%%%%%%%%%%%%%%%
\end{document}